\documentclass[11pt]{article}
\usepackage{amsmath,amssymb,amsfonts,amsthm}
\usepackage{mathrsfs}
\usepackage{indentfirst}
\usepackage[pdftex]{color,graphicx}
\usepackage{subfig}
\usepackage[pdftex,bookmarks,unicode,colorlinks]{hyperref}
\usepackage{enumerate} 
\usepackage[numbers]{natbib} 
\usepackage{yfonts} 
\usepackage{caption} 

\theoremstyle{plain}
\newtheorem{theorem}{Theorem}[section]
\newtheorem{lemma}[theorem]{Lemma}
\newtheorem{corollary}[theorem]{Corollary}

\theoremstyle{remark}

\theoremstyle{definition}

\newcommand{\R}{\mathbb R}

\newcommand{\e}{\epsilon}

\DeclareMathOperator*{\argmin}{arg\,min}

\numberwithin{equation}{section} 

\begin{document}

\title{\bf\Large Large Deviations for Stochastic Differential Equations Driven by Heavy-tailed L\'evy Processes }
\author{\bf\normalsize{
Wei Wei$^{1,}$\footnote{Email: \texttt{weiw16@hust.edu.cn}},
Qiao Huang$^{2,}$\footnote{Email: \texttt{qhuang@fc.ul.pt}},
Jinqiao Duan$^{3,}$\footnote{Email: \texttt{duan@iit.edu}}
} \\[10pt]
\footnotesize{$^1$Center for Mathematical Sciences, Huazhong University of Science and Technology,} \\
\footnotesize{Wuhan, Hubei 430074, China.} \\[5pt]
\footnotesize{$^2$Department of Mathematics, Faculty of Sciences, University of Lisbon,}\\
\footnotesize {Campo Grande, Edif\'{\i}cio C6, PT-1749-016 Lisboa, Portugal.} \\[5pt]
\footnotesize{$^3$Department of Applied Mathematics, Illinois Institute of Technology, Chicago, IL 60616, USA.}
}

\date{}
\maketitle
\vspace{-0.3in}

\begin{abstract}
 We obtain sample-path large deviations for a class of one-dimensional stochastic differential equations with bounded drifts and heavy-tailed L\'evy processes. These heavy-tailed L\'evy processes do not satisfy the exponential integrability condition, which is a common restriction on the L\'evy processes in existing large deviations contents. We further prove that the solution processes satisfy a weak large deviation principle with a discrete rate function and logarithmic speed. We also show that they do not satisfy the full large deviation principle.
\end{abstract}

\section{Introduction}

Stochastic differential equations (SDEs) driven by Brownian motion have proven their power in many fields \cite{Henry,Oksendal,Skorokhod,DuanS}.
However, random fluctuations in complex systems in science are often non-Gaussian \cite{Woy, Del1, Del2}. Especially, the $\alpha$-stable L\'evy motions are thought to be a good substitution of Brownian motion.

We will investigate the large deviations for a one-dimensional SDE driven by a class of heavy-tailed L\'evy processes,
\begin{equation}\label{SDE}
  \begin{cases}
  dY_t^\epsilon = b(Y_t^\epsilon)dt+\epsilon d L_t^\epsilon, & \quad t \in (0,1], \\
  Y_0^\epsilon = 0, & 
  \end{cases}
\end{equation}
where
\[
\epsilon L_t^\epsilon=\sqrt{\epsilon}B_t+ \epsilon \int_0^t\int_{\mathbb{R}\backslash \{0\}}z \widetilde{N}^{\frac{1}{\epsilon}}(ds,dz),
\]
and $\widetilde{N}^{\frac{1}{\epsilon}}$ is a compensated Poisson random measure defined on a given complete probability space $(\Omega,\mathcal{F}, \mathbb{P})$, with compensator $\epsilon^{-1}ds \otimes \nu$. The measure $\nu$ is a L\'evy measure that will be specified latter.

There have been active studies and applications of the dynamical behaviors of such systems. The most probable transition path \cite{ Weinan, Wan}, the mean exit time \cite{Ting, Imkeller} and so forth are used in identifying these behaviors. The large deviation principle has become a massive tool in understanding these deterministic properties as is shown in Freidlin and Wentzell \cite{FW}. Roughly speaking, the large deviation principle deals with the identification of asymptotic exponential decay rate of probabilities. In the classical framework, a sequence of random elements $\{X_n\}_{n \geq 1}$ valued in some Polish space $\mathcal{E}$ is said to satisfy the large deviation principle with rate function $I$ and speed $k(n)$, if
 \[
   -\inf_{x \in A^\circ} I(x) \leq \liminf_{n \to \infty} \frac{\log \mathbb{P}(X_n \in A)}{k(n)} \leq \limsup_{n \to \infty} \frac{\log \mathbb{P}(X_n \in A)}{k(n)} \leq -\inf_{x \in \bar{A}} I(x)
 \]
for every Borel measurable set $A$ in $\mathcal{E}$.

Many works are about identifying whether the solution processes of SDEs driven by Brownian motion satisfy a large deviation principle both in finite dimension and in infinite dimension \cite{Demb, Dupuis,Budhiraja08}. Recently, Budhiraja et al.~\cite{Budhiraja11,Budhiraja13} have obtained a large deviation principle for SDEs driven by random Poisson measures from finite dimensional settings to infinite dimensional settings, using the weak convergence approach. The relations between uniformly exponential tightness and the large deviation principle also give rise to results on SDEs driven by semimartingales \cite{Ganguly18, Qiao}.

However, an \emph{exponential integrability condition} on the L\'evy measure $\nu $ is unavoidable in these works, which is
\[
\int_{\mathbb{R}^d}\textrm{e}^{\lambda |x|} \nu(dx)<\infty, \quad \textrm{for every } \lambda>0.
\]
Note that the $\alpha$-stable L\'evy processes do not satisfy this condition, and thus the exponentially light tailed L\'evy processes are alternatively used as the driven noise in some articles investigating large deviations principles of SDEs \cite{Gomes}.
In the classical approach to obtain large deviations principles, this condition is necessary because their rate functions are determined by the Laplace transform of the processes through the Legendre transform.

In the present paper, we only require the L\'evy processes to have regularly varying L\'evy measure $\nu$, that is,
\begin{equation}\label{Connu}
  \nu([x,\infty)) = L_+(x)x^{-\alpha}, \quad \nu((-\infty,-x]) = L_-(x)x^{-\beta}, \quad \forall x>0.
\end{equation}
$L_+$ and $L_-$ are two slowly varying functions which means $L_\pm >0$ and $\lim_{u \to +\infty} L_\pm(\lambda u )/L_\pm(u)=1$ for every $\lambda >0$. We assume the constants $\alpha,\beta >1$ in this paper.

We can see that the $\alpha$-stable L\'evy processes do not fit in the classical large deviations framework. So, recently, Rhee et al.~\cite{CH19} have proved a large deviation result on such one-dimensional regularly varying L\'evy processes by the $\mathbb{M}$-convergence \cite{Lindskog}. They also obtained a weak large deviation principle in the classical framework and deduced that a full large deviation principle does not hold by a counterexample. At first glance, the large deviations for SDEs driven by such L\'evy processes may seem to be an immediate consequence of the \emph{Contraction Principle}. However, as the rate function they obtained is not good, the \emph{Contraction Principle} is is no longer applicable in this case. Moreover, as their results are obtained in the Skorokhod space, the solution mapping should be verified to be continuous under the Skorokhod metric. As an intermediate step, a ``bounded away'' condition requires a careful treatment as well.

This paper is organized as follows. After an introduction to basic notions and results obtained by Rhee et al. in Section \ref{pLDP}, we present our main results in Section \ref{sLDP}. Theorem \ref{LDP} shows that a kind of large deviations estimates hold for every bounded measurable sets in the Skorokhod space. Theorem \ref{wLDP} presents a weak large deviation principle. In subsection \ref{nexist}, we show the full large deviation principle does not hold. We discuss our results and its further applications in Section \ref{Dis}. The proofs to three key lemmas are left in Appendix.

\section{Large Deviations for L\'evy processes}\label{pLDP}
This section reviews general concepts in the large deviation theory, and recalls useful results of Rhee et al.~\cite{CH19}. These results are the groundwork for deriving our main theorem.

Let $\mathbb{D}=\mathbb{D}([0,1])$ be the Skorokhod space on $[0,1]$, that is, the space of real-valued right continuous functions with left limits. We use the usual Skorokhod $J_1$ metric $d(x,y)=\inf_{\lambda\in \Lambda}\|\lambda-e \| \vee \|x \circ \lambda - y \|$, where $ \Lambda$ denotes the set of increasing homeomorphisms from $[0,1]$ to itself, $e$ is the identity and $\|\cdot\|$ is the uniform norm. We say that a set $A \subset \mathbb{D}$ is bounded away from another set $B \subset \mathbb{D}$ if $\inf_{x\in A, y \in B}d(x,y)>0$.

Let $\mathcal{G}$ be the Borel $\sigma$-algebra on $(\mathbb{D},d)$. Given a closed subset $\mathbb{S}$ of $\mathbb{D}$, let $\mathbb{D} \backslash \mathbb{S}$ be equipped with the relative topology as a subspace of $\mathbb{D}$. Consider the associated Borel $\sigma$-algebra $\mathcal{G}_{\mathbb{D}\backslash \mathbb{S}}$ on $\mathbb{D} \backslash \mathbb{S}$. Then it is easy to see that $\mathcal{G}_{\mathbb{D}\backslash \mathbb{S}}$ is just the restriction of $\mathcal{G}$ on $\mathbb{D} \backslash \mathbb{S}$, i.e., $\mathcal{G}_{\mathbb{D}\backslash \mathbb{S}} = \{A \in \mathcal{G}: A \subset \mathbb{D}\backslash \mathbb{S} \}$. Denote $\mathbb{S}^r \triangleq \{ x \in \mathbb{D}: d(x,\mathbb{S})<r \}$ for $r>0$. Let $\mathbb{M}(\mathbb{D}\backslash \mathbb{S})$ be the class of measures defined on $ \mathcal{G}_{\mathbb{D}\backslash \mathbb{S}}$ whose restrictions to $\mathbb{D}\backslash \mathbb{S}^r$ are finite for all $r>0$. Topologize $\mathbb{M}(\mathbb{D}\backslash \mathbb{S})$ with a subbasis $\{ \{\mu \in \mathbb{M}(\mathbb{D}\backslash \mathbb{S}): \mu(f) \in G\}: f \in \mathcal{C}_{\mathbb{D}\backslash \mathbb{S}}, G \textrm{ open in } \mathbb{R}_+\}$ where $\mathcal{C}_{\mathbb{D}\backslash \mathbb{S}}$ is the set of real-valued, non-negative, bounded, continuous functions whose supports are bounded away from $\mathbb{S}$.

Let $\mathbb{D}_{l,m}$ be the subspace of $\mathbb{D}([0,1])$ consisting of step functions that vanish at the origin with $l$ upward jumps and $m$ downward jumps. For given $\alpha,\beta>1$ and $(j,k)\in \mathbb{Z}^2_+$, define $\mathbb{D}_{<j,k} \triangleq \cup_{(l,m)\in \mathbb{I}_{<j,k}} \mathbb{D}_{l,m}$, where $\mathbb{I}_{<j,k} \triangleq \{ (l,m)\in \mathbb{Z}^2_+\backslash \{(j,k)\}: (\alpha-1)l+(\beta-1)m \leq (\alpha-1)j+(\beta-1)k\}$ and $\mathbb{Z}_+$ denotes the set of non-negative integers.

Consider a scaled one-dimensional L\'evy process $\overline{L_t^n}$,
\[
  \overline{L^n_t} = \frac{1}{n} B(nt)+ \frac{1}{n}\int_{\mathbb{R}\backslash \{ 0\}} z \widetilde{N}([0,nt]\times dz),
\]
where $\widetilde{N}$ is a compensated Poisson random measure defined on a given complete probability space $(\Omega,\mathcal{F}, \mathbb{P})$ with compensator $ds \otimes \nu$. Here, L\'evy measure $\nu$ satisfies equations (\ref{Connu}).

Denote $C_{0,0}\triangleq \delta_{\mathbf{0}}(\cdot)$, the Dirac measure concentrated on the zero function. Let $\nu_\alpha(x,\infty) \triangleq x^{-\alpha} $ be a measure concentrated on $(0,\infty)$. For each $(j,k) \in \mathbb{Z}^2_+ \backslash \{(0,0) \}$, let $\nu_\alpha^j$ and $\nu_\beta^k$ be respectively the $j$-fold product of $\nu_\alpha$ and the $k$-fold product of $\nu_\beta$. Define a measure $C_{j,k} \in \mathbb{M}(\mathbb{D}\backslash \mathbb{D}_{<j,k})$ concentrated on $\mathbb{D}_{j,k}$ as
\[\textstyle{C_{j,k}(\cdot) \triangleq \mathbb{E}\left[ (\nu_\alpha^j \times \nu_\beta^k) \left( \{ (x,y) \in (0,\infty)^j \times (0,\infty)^k: \sum_{i=1}^j x_i 1_{[U_i,1]}-\sum_{i=1}^k y_i 1_{[V_i,1]} \in \cdot\} \right)\right],}
\]
where $U_i$'s and $V_i$'s are i.i.d. random variables uniformly distributed on $[0,1]$.

Define $\mathcal{I}(j,k) \triangleq (\alpha-1)j+(\beta-1)k$. For each $A \in \mathcal{G}$, consider a pair of integers $ (\mathcal{J}(A),\mathcal{K}(A))$ such that
\begin{equation} \label{CJK}
  (\mathcal{J}(A),\mathcal{K}(A)) \in \argmin \limits_{ \begin{subarray}{c}  (j,k)\in\mathbb{Z}_+^2, \\ \mathbb{D}_{j,k}\cap A \neq \emptyset\end{subarray}} \mathcal{I}(j,k).
\end{equation}

Rhee et al.~proved the following large-deviation theorem for $\overline{L^n}$.
\begin{theorem}[{\cite[Theorem 3.4]{CH19}}]\label{CH-LDP}
  Suppose that $A$ is a measurable set in $\mathcal{G}$. If the argument minimum in (\ref{CJK}) is non-empty and $A$ is bounded away from $\mathbb{D}_{<\mathcal{J}(A),\mathcal{K}(A)}$, the argument minimum is then unique and
  \begin{align*}
    \liminf_{n \to \infty} \frac{\mathbb{P}(\overline{L^n}\in A)}{(n \nu[n,\infty))^{\mathcal{J}(A)}(n \nu(-\infty,-n])^{\mathcal{K}(A)}} & \geq C_{\mathcal{J}(A),\mathcal{K}(A)}(A^\circ), \\
    \limsup_{n \to \infty} \frac{\mathbb{P}(\overline{L^n} \in A)}{(n\nu[n,\infty))^{\mathcal{J}(A)}(n\nu(-\infty,-n])^{\mathcal{K}(A)}} & \leq C_{\mathcal{J}(A),\mathcal{K}(A)}(\bar{A}).
  \end{align*}
  Moreover, if the argument minimum in (\ref{CJK}) is empty and $A$ is bounded away from $\mathbb{D}_{<l,m} \cup \mathbb{D}_{l,m} $ for some $ (l,m) \in \mathbb{Z}_+^2 \backslash \{(0, 0)\}$, we have
  \begin{equation*}
    \lim_{n \to \infty} \frac{\mathbb{P}(\overline{L^n}\in A)}{(n \nu[n,\infty))^l (n\nu(-\infty,-n])^m} =0.
  \end{equation*}
\end{theorem}

Define a rate function $I: \mathbb D\to [0,\infty]$ as follows (it is indeed a rate function as shown in \cite{CH19}),
\begin{equation*}
  I(\xi) \triangleq
  \begin{cases}
    (\alpha-1)\mathcal{D}_+(\xi)+(\beta-1)\mathcal{D}_-(\xi), & \textrm{if }\xi \textrm{ is a step function and } \xi(0)=0, \\
    \infty, & \textrm{otherwise},
  \end{cases}
\end{equation*}
where $\mathcal{D}_+(\xi)$ counts the upward jumps of $\xi$ and $\mathcal{D}_-(\xi)$ counts the downward jumps of $\xi$.

Based on Theorem \ref{CH-LDP}, Rhee et al. also proved the following weak large deviation principle for $\overline{L^n}$.
\begin{theorem}[{\cite[Theorem 4.2]{CH19}}]\label{CH-WLDP}
  The scaled process $\overline{L^n}$ satisfies the weak large deviation principle with rate function $I$ and speed $\log n$, i.e.,
  \begin{align*}
     \liminf_{n \to \infty} \frac{\log \mathbb{P}(\overline{L^n} \in G)}{\log n} & \geq -\inf_{x \in G} I(x),\\
     \limsup_{n \to \infty} \frac{\log \mathbb{P}(\overline{L^n} \in K)}{\log n} & \leq -\inf_{x \in K} I(x),
  \end{align*}
  for every open set $G$ and compact set $K$ in $\mathbb{D}$.
\end{theorem}

\section{Large Deviations for One-dimensional Stochastic Differential Equations} \label{sLDP}
In this section, we will prove that solutions to a kind of one-dimensional SDEs that are driven by heavy-tailed L\'evy processes, also satisfy the large-deviation theorem with the rate function transformed from the one for driven processes by a solution mapping.
\subsection{Large deviations}
According to \cite{priola2012}, we assume that $b$ is a bounded Lipschitz continuous function on $\R$ to ensure pathwise unique solutions of equations (\ref{SDE}).

Let $\epsilon=1/n$. Then, in distribution sense, we have
\begin{equation*}
  \epsilon L_t^\epsilon =\epsilon \frac{1}{\sqrt{\epsilon}} B_{t}+ \epsilon \int_0^t\int_{\mathbb{R}\backslash \{0\}}z \widetilde{N}^{\frac{1}{\epsilon}}(ds,dz) \overset{d}{=} \frac{1}{n} B_{nt}+\frac{1}{n}\int_0^t\int_{\mathbb{R}\backslash \{0\}}z \widetilde{N}(nds,dz)
  =\frac{1}{n}L_{nt}.
\end{equation*}

From Theorem \ref{CH-LDP} and Theorem \ref{CH-WLDP}, we know that the processes $\{\epsilon L^\epsilon \}_{\epsilon \geq 0}$ satisfy the large deviations therein.

Define a deterministic mapping $F: \mathbb{D} \to \mathbb{D}$ by $f=F(g)$, where $f$ is the solution of
\begin{equation}\label{F}
  f(t)=\int_0^t b(f(s))ds+g(t), \quad t \in [0,1].
\end{equation}

For every fixed $g\in \mathbb D$, one can easily verify that the following mapping (which should not be confused with $F$)
\begin{equation}\label{contract}
  f\mapsto \int_0^\cdot b(f(s))ds+g
\end{equation}
is a contraction on the space $\mathbb D([0,\delta])$ under the uniform metric, with a contraction constant strictly less than 1 for sufficiently small $\delta>0$. Thus, the mapping \eqref{contract} has a unique fixed-point $f$, which solves the equation \eqref{F} on the interval $[0,\delta]$, due to the Banach fixed-point theorem. By a standard extension procedure, one can extend the solution to the whole interval $[0,1]$. Since the definition of fixed-points is independent of the choice of metrics on $\mathbb D$, once the fixed-point $f$ uniquely exists by the contraction mapping \eqref{contract} under the uniform metric, it will also uniquely exist under the Skorokhod $J_1$ metric (in fact, it will always uniquely exist whatever metric we use). Indeed, the existence is obvious, and if there is another fixed-point $f^*$, then by the assumption, $\|f-f^*\|=0$, and hence $d(f,f^*)=0$.

%

Thus, $F$ is a well-defined mapping from $\mathbb{D}$ to $\mathbb{D}$. If we integrate (\ref{SDE}) over $[0,t]$, we see that $Y^\epsilon_t=F(\epsilon L^\epsilon)(t)$.

Inspired by \emph{Contraction Principle}, we want to prove that $F$ is continuous under the Skorokhod metric $d$. It is done in the following lemma.

\begin{lemma}\label{Conti}
  $F$ is a continuous mapping under Skorokhod metric $d$.
\end{lemma}

The proof is provided in Appendix.


In light of Theorem \ref{CH-LDP} and the continuity of $F$, the large-deviations results for the solution processes boil down to the driven processes. The result is immediate once we have proved a ``bounded away" property for every set considered therein.

For this purpose, we define $\mathbb{G}_{<j,k}\triangleq \overline{F(\mathbb{D}_{<j,k})}$ for every pair $(j,k) \in \mathbb{N}^2$. We provide a lemma below to ensure $F^{-1}(A)$ being bounded away from $\mathbb{D}_{<j,k}$.

\begin{lemma}\label{bda}
For every bounded measurable set $A$ in $\mathbb{D}$ which is bounded away from $\mathbb{G}_{<j,k}$, the pre-image $F^{-1}(A)$ is bounded away from $\mathbb{D}_{<j,k}$. If $A$ is also bounded away from $\overline{F(\mathbb{D}_{<j,k}\cup\mathbb{D}_{j,k})}$, the pre-image $F^{-1}(A)$ is bounded away from $\mathbb{D}_{<j,k}\cup\mathbb{D}_{j,k}$.
\end{lemma}

The proof is similar to Lemma \ref{Conti}. It is provided in Appendix.

The following simple lemma will simplify our results.

\begin{lemma}\label{Inj}
  The mapping $F$ is an injection.
\end{lemma}

\begin{proof}

Proof by contradiction.

Suppose there exist $\xi_1$, $\xi_2$ and $d(\xi_1,\xi_2) >0$ but  $F(\xi_1)=F(\xi_2)$. We know that $\|\xi_1-\xi_2\|>0$. Then by the definition of $F$, we know
\begin{align*}
F(\xi_1)(t)&=\int_0^tb(F(\xi_1)(s))ds+\xi_1(t), \\
F(\xi_2)(t)&=\int_0^tb(F(\xi_2)(s))ds+\xi_2(t).
\end{align*}

Therefore, $\|F(\xi_1)-F(\xi_2) \|=\|\xi_1-\xi_2\|>0$ contradicts with the assumption that $F(\xi_1)=F(\xi_2)$.
\end{proof}

So, we know that $F^{-1}(\xi)$ is uniquely determined by every $\xi$ in $F(\mathbb{D}([0,1]))$. The following lemma is essential in establishing the weak large deviation principle.

Endow the image $F(\mathbb{D})$ with the same metric as $\mathbb{D}$. The following shows that when restricted on $F(\mathbb{D})$, the inverse mapping $F^{-1}$ is also continuous. The proof is similar to Lemma \ref{Conti}, and we leave it into Appendix.

\begin{lemma}\label{InConti}
  The inverse mapping $F^{-1}: F(\mathbb{D}) \to \mathbb{D}$ is a continuous mapping under Skorokhod metric $d$. Moreover, $F^{-1}$ is Cauchy continuous, which means $F^{-1}$ maps all the Cauchy sequences in $(F(\mathbb{D}),d)$ to Cauchy sequences in $(\mathbb{D},d)$.
\end{lemma}

Now combining Lemma \ref{Conti}, \ref{Inj} and \ref{InConti}, we know that $F: \mathbb{D} \to \mathbb{D}$ is a topological embedding, whose inverse $F^{-1}$ is Cauchy continuous when restricted on $F(\mathbb{D})$.

For every Cauchy sequence $\{\xi_n\}_{n\geq 1}$ in $(F(\mathbb{D}),d)$, we know that by the Cauchy continuity of $F^{-1}$, $\{F^{-1}(\xi_n)\}_{n \geq 1}$ is a Cauchy sequence in $(\mathbb{D},d)$. Then, the completeness of $\mathbb D$ yields that there exists $g \in \mathbb{D}$ such that $\lim_{k \to \infty} d(F^{-1}(\xi_n),g) =0$. As a consequence of the continuity of $F$, $\lim_{k \to \infty} d(\xi_n,F(g)) =0$, which means that $F(\mathbb{D})$ is a closed subset of $\mathbb{D}$. More generally, let $A \subset \mathbb{D}$ be a closed subset. Then $(A, d)$ is a complete metric subspace of $\mathbb{D}$. Since $F^{-1}|_{F(\mathbb{D})}$ is Cauchy-continuous, the restriction $F^{-1}|_{F(A)}$ is also Cauchy-continuous. A similar argument yields that $F(A)$ is closed in $\mathbb{D}$. 

Furthermore, since $F$ is a topological embedding with closed image, $F$ is also a proper mapping (see e.g. \cite[Proposition A.53]{Lee13}), that is, the pre-image of every compact set is compact. Indeed, let $K\subset \mathbb{D}$ be compact. Since $F(\mathbb{D})$ is closed, we have that $K \cap F(\mathbb{D})$ is compact. Then the continuity of $F^{-1}: F(\mathbb{D}) \to \mathbb{D}$ in Lemma \ref{InConti} implies that $F^{-1}(K \cap F(\mathbb{D})) = F^{-1}(K)$ is compact.

Therefore, we have the following corollary.

\begin{corollary}\label{Closed}
  The mapping $F: \mathbb{D} \to \mathbb{D}$ is a closed mapping, which means $F$ maps closed subsets of $(\mathbb{D},d)$ to closed subsets. The mapping $F$ is also a proper mapping, which means the pre-image of every compact set by $F$ is compact.
\end{corollary}

Define
\begin{equation}\label{JK}
(\widetilde{\mathcal{J}}(A),\widetilde{\mathcal{K}}(A)) \in \argmin \limits_{ \begin{subarray}{c}  (j,k)\in\mathbb{Z}_+^2, \\ F(\mathbb{D}_{j,k})\cap A \neq \emptyset\end{subarray}} \mathcal{I}(j,k),
\end{equation}
and
\begin{equation*}
\widetilde{C}_{j,k}(A)={C}_{j,k}(F^{-1}(A)).
\end{equation*}

As the mapping $F$ is an injection, we know that
\[
  F^{-1} (F(\mathbb{D}_{j,k}) \cap A)= \mathbb{D}_{j,k} \cap F^{-1}(A).
\]
This gives
\[
  \{(j,k)\in\mathbb{Z}_+^2:F(\mathbb{D}_{j,k})\cap A \neq \emptyset\}=\{(j,k)\in\mathbb{Z}_+^2: \mathbb{D}_{j,k}\cap F^{-1}(A)  \neq \emptyset\}.
\]
Thus, we have
\begin{equation}\label{ejk}
    \argmin \limits_{ \begin{subarray}{c}  (j,k)\in\mathbb{Z}_+^2, \\ F(\mathbb{D}_{j,k})\cap A \neq \emptyset\end{subarray}} \mathcal{I}(j,k)=\argmin \limits_{ \begin{subarray}{c}  (j,k)\in\mathbb{Z}_+^2, \\ \mathbb{D}_{j,k}\cap F^{-1}(A)  \neq \emptyset\end{subarray}} \mathcal{I}(j,k),
\end{equation}
which means when one side of the above equation is unique, the other is unique too.

We can now deduce our main large deviation theorem for bounded measurable sets in $\mathbb{D}$.
\begin{theorem}[Large deviations] \label{LDP}
  Suppose that $A$ is a bounded measurable set. If the argument minimum in (\ref{JK}) is non-empty and $A$ is bounded away from $\mathbb{G}_{<\widetilde{\mathcal{J}}(A),\widetilde{\mathcal{K}}(A)}$, the argument minimum is then unique and
\begin{align*}
\liminf_{\epsilon \to 0} \frac{\mathbb{P}(Y^\epsilon \in A)}{(1/\epsilon \nu[1/\epsilon,\infty))^{\widetilde{\mathcal{J}}(A)}(1/\epsilon \nu(-\infty,-1/\epsilon])^{\widetilde{\mathcal{K}}(A)}} & \geq \widetilde{C}_{\widetilde{\mathcal{J}}(A),\widetilde{\mathcal{K}}(A)}(A^\circ), \\
\limsup_{\epsilon \to 0} \frac{\mathbb{P}(Y^\epsilon \in A)}{(1/\epsilon\nu[1/\epsilon,\infty))^{\widetilde{\mathcal{J}}(A)}(1/\epsilon\nu(-\infty,-1/\epsilon])^{\widetilde{\mathcal{K}}(A)}} & \leq \widetilde{C}_{\widetilde{\mathcal{J}}(A),\widetilde{\mathcal{K}}(A)}(\bar{A}). \\
\end{align*}
Moreover, if the argument (\ref{JK}) is empty and $A$ is bounded away from $\overline{F(\mathbb{D}_{<l,m}\cup\mathbb{D}_{l,m})}$ for some $(l,m) \in \mathbb{Z}^2_+\backslash\{(0,0)\}$, we have
\begin{equation}\label{LDP-3}
  \lim_{\epsilon \to 0} \frac{\mathbb{P}(Y^\epsilon\in A)}{[1/\epsilon \nu[1/\epsilon,\infty)]^l[1/\epsilon\nu(-\infty,-1/\epsilon]]^m}=0.
\end{equation}

\end{theorem}

\begin{proof}

  We first show that $F^{-1}(A)$ is bounded away from $\mathbb{D}_{<\mathcal{J}(F^{-1}(A)),\mathcal{K}(F^{-1}(A))}$.

  By Lemma \ref{bda}, we know that $F^{-1}(A)$ is bounded away from $\mathbb{D}_{<\widetilde{\mathcal{J}}(A),\widetilde{\mathcal{K}}(A)}$. As equation (\ref{ejk}) holds for the set $A$, all possible choice of the two pairs $(\widetilde{\mathcal{J}}(A),\widetilde{\mathcal{K}}(A))$ and $(\mathcal{J}(F^{-1}(A)),\mathcal{K}(F^{-1}(A))$ share the same subset of $\mathbb{Z}^2_+$. Then, we know $F^{-1}(A)$ is bounded away from $\mathbb{D}_{<\mathcal{J}(F^{-1}(A)),\mathcal{K}(F^{-1}(A))}$. So the right hand side of equation (\ref{ejk}) is unique, then $ (\widetilde{\mathcal{J}}(A),\widetilde{\mathcal{K}}(A))=(\mathcal{J}(F^{-1}(A)),\mathcal{K}(F^{-1}(A))$.

  Applying Theorem \ref{CH-LDP} to the set $F^{-1}(A)$, we get
  \begin{align}
    \liminf_{\epsilon \to 0} \frac{\mathbb{P}(\epsilon L^\epsilon \in F^{-1}(A))}{(1/\epsilon \nu[1/\epsilon,\infty))^{\mathcal{J}(F^{-1}(A))}(1/\epsilon \nu(-\infty,-1/\epsilon])^{\mathcal{K}(F^{-1}(A))}} & \geq C_{\mathcal{J}(F^{-1}(A)),\mathcal{K}(F^{-1}(A))}\left((F^{-1}(A))^\circ\right), \label{ldp1}\\
    \limsup_{\epsilon \to 0} \frac{\mathbb{P}(\epsilon L^\epsilon \in F^{-1}(A))}{(1/\epsilon\nu[1/\epsilon,\infty))^{\mathcal{J}(F^{-1}(A))}(1/\epsilon\nu(-\infty,-1/\epsilon])^{\mathcal{K}(F^{-1}(A))}} & \leq C_{\mathcal{J}(F^{-1}(A)),\mathcal{K}(F^{-1}(A))}(\overline{F^{-1}(A)}). \label{ldp2}
  \end{align}
  Note that $F$ is continuous under $d$, then we know that $F^{-1}(A^\circ) \subset (F^{-1}(A))^\circ$ and $\overline{F^{-1}(A)} \subset F^{-1}(\bar{A})$. Replacing $(\mathcal{J}(F^{-1}(A)),\mathcal{K}(F^{-1}(A))$ by $(\widetilde{\mathcal{J}}(A),\widetilde{\mathcal{K}}(A))$ and $C$ by $\tilde{C}$ in (\ref{ldp1}) and (\ref{ldp2}), we get the desired result.

  For equation (\ref{LDP-3}), we know that $F^{-1}(A)$ is bounded away from $\mathbb{D}_{<l,m}\cup\mathbb{D}_{l,m}$ by Lemma \ref{bda}. Applying Theorem \ref{CH-LDP} to $F^{-1}(A)$, we then get (\ref{LDP}).
\end{proof}

We will also obtain a weak large deviation principle as follows.

Define a function $\tilde{I}: \mathbb D\to [0,\infty]$ by
\[
\tilde{I}(\xi) \triangleq
\begin{cases}
  (\alpha-1)\mathcal{D}_+(\xi)+(\beta-1)\mathcal{D}_-(\xi), & \textrm{if } \xi=F(\eta) \textrm{ and }\eta \textrm{ is a step function with } \eta(0)=0, \\
\infty, & \textrm{otherwise,}
\end{cases}
\]
where $\mathcal{D}_+(\xi)$ counts the upward jumps of $\xi$ and $\mathcal{D}_-(\xi)$ counts the downward jumps of $\xi$. 

For each $\xi\in F(\mathbb D)$, as $F^{-1}(\xi )$ and $\xi $ share the same jump times and jump sizes, by the definition of $\tilde{I}$, one can easily see that $\tilde{I}(\xi)=I(F^{-1}(\xi))$. In other words, $\tilde{I}$ is the ``pushforward'' of $I$ by $F$. Now, since $F$ is a closed mapping by Corollary \ref{Closed}, we know that $\tilde I$ is also a rate function as the level sets of $\tilde I$ are just the images of the level sets of $I$ under $F$ (cf. the proof of \cite[Theorem 4.2.1]{Demb}).

\begin{theorem}[Weak large deviation principle]\label{wLDP}
  The solution $Y^\epsilon$ satisfies the weak large deviation principle with the rate function $\tilde{I}$ and speed $ \log 1/\epsilon$, i.e.,
\begin{equation*}
-\inf_{x \in G} \tilde{I}(x) \leq \liminf_{\epsilon \to 0} \frac{\log \mathbb{P}(Y^\epsilon \in G)}{\log 1/\epsilon},
\end{equation*}
 and
\begin{equation}\label{wldp2}
\limsup_{\epsilon \to 0} \frac{\log \mathbb{P}(Y^\epsilon \in K)}{\log 1/\epsilon} \leq -\inf_{x \in K} \tilde{I}(x),
\end{equation}
for every open set $G$ and compact set $K$ in $\mathbb{D}$.
\end{theorem}

\begin{proof}

With the fact that $F^{-1}(K)$ is compact (by Corollary \ref{Closed}) and $F^{-1}(G)$ is open for every open set $G\subset \mathbb D$ and every compact set $K\subset \mathbb D$, and  $\mathbb{P}(Y^\epsilon\in \cdot)=\mathbb{P}(\epsilon L^\epsilon \in F^{-1}(\cdot))$, we apply Theorem \ref{CH-WLDP} to $F^{-1}(G)$ and $F^{-1}(K)$ respectively. We then get
\[
-\inf_{x \in F^{-1}(G)} I(x) \leq \liminf_{\epsilon \to 0} \frac{\log \mathbb{P}(Y^\epsilon \in G)}{\log 1/\epsilon},
\limsup_{\epsilon \to 0} \frac{\log \mathbb{P}(Y^\epsilon \in K)}{\log 1/\epsilon} \leq -\inf_{x \in F^{-1}(K)} I(x).
\]
By the definition of $F$, we know that $F$ preserves the jump times and jump sizes.

The proof is completed once we have proved the following claim that
\begin{equation*}
 \inf_{x\in F^{-1}(U)} I(x) =\inf_{x\in U} \tilde{I}(x),
\end{equation*}
for every measurable set $U$.

Assume that $\inf_{x\in U} \tilde{I}(x)<\infty$. Then, there exists a function $\xi \in U$ such that $F^{-1}(\xi)$ is a step function with finite jumps and starting at 0. For such $\xi$, $F^{-1}(\xi)\in F^{-1}(U)$. As we have seen, $\tilde{I}(\xi)=I(F^{-1}(\xi))$. Then,
\[
\inf_{x\in F^{-1}(U)} I(x)\leq\inf_{x\in U} \tilde{I}(x).
\]

A similar discussion will show that
\[
\inf_{x\in F^{-1}(U)} I(x)\geq\inf_{x\in U} \tilde{I}(x).
\]

If $\inf_{x\in U} \tilde{I}(x)=\infty$, we know that for all $\xi \in F^{-1}(U)$, $I(\xi)=\infty$. Similarly, if $\inf_{x\in F^{-1}(U)} I(x)=\infty$, we know that for all $\xi \in U$ , $\tilde{I}(\xi)=\infty$.

Thus, we have proved our claim and this completes the proof of the theorem.
\end{proof}

\subsection{Nonexistence of strong large deviation principle}\label{nexist}

Consider a mapping $\pi: \mathbb{D}^2 \to \mathbb{R}^2_+$ that maps paths in $\mathbb{D}$ to their largest jump sizes, i.e.
\[
\pi(\xi) \triangleq \left( \sup_{t\in (0,1]} (\xi(t)-\xi(t^-)),\sup_{t\in (0,1]} (\xi(t^-)-\xi(t)) \right).
\]
We know by Rhee et al.~\cite[Section 4.4]{CH19} that $\pi$ is continuous under $d$.

Prove by contraction. Suppose $Y^\epsilon$ satisfies a strong large deviation principle. That means we suppose (\ref{wldp2}) to hold for all closed sets rather than just compact sets.

The random variables $\pi(Y^\epsilon)$ have to satisfy a strong large deviation principle with the rate function
\[
\tilde{I}'(y)=\inf\{\tilde{I}(\xi): \xi \in \mathbb{D}, y=\pi(\xi)\}
\]
by the \emph{Contraction Principle}, as long as $\tilde{I}'$ is proved to be a rate function \cite[Remarks of Theorem 4.2.1]{Demb}.

As $F$ is an injection, $F^{-1}(\xi)$ is a function in $\mathbb{D}$ and $I(F^{-1}(\xi))=\tilde{I}(\xi)$ by taking the measurable set as $\{\xi\}$ for every $\xi \in F(\mathbb{D})$.

For the reason that  $F$ preserves the jump times and jump sizes of the L\'evy process $\epsilon L^\epsilon$, we know by the definition of $\pi $ that
\[
\pi(Y^\epsilon)=\pi(F(\epsilon L^\epsilon))=\pi(\epsilon L^\epsilon).
\]

We then know that  $\pi (\xi)=\pi(F^{-1}(\xi))$ for every $\xi \in F(\mathbb{D})$ due to the fact that $\xi =F(F^{-1}(\xi))$.

Then, we conclude that
\[
\tilde{I}'(y)=\inf\{I(F^{-1}(\xi)): \xi \in \mathbb{D}, y=\pi(\xi)\}=\inf\{I(F^{-1}(\xi)): \xi \in \mathbb{D}, y=\pi(F^{-1}(\xi)\}=I'(y),
\]
where $I'$ is the rate function defined in \cite[Section 4.4]{CH19}. Thus $\tilde{I}'$ indeed is a rate function.

So, $\epsilon L^\epsilon$ should satisfy the strong large deviation principle with the rate function $I'$, which is proved incorrect with a counterexample in \cite[Section 4.4]{CH19}.

\subsection{An example}
In this subsection we will show that the solution process to an SDEs driven by $\alpha$-stable L\'evy process satisfies previous large-deviations estimates.

Let $L_t$ be a symmetric $\alpha$-stable L\'evy process, with $1<\alpha<2$. Then the scaled process $\epsilon L_{t/\epsilon}$ admits the L\'evy-It\^o decomposition that
\begin{equation} \label{LIDeco}
  \epsilon L_{t/\epsilon} = \epsilon \int_0^t \int_{[-1,1] \backslash \{0\}} z \widetilde{N}^{\frac{1}{\epsilon}}(ds,dz)+ \epsilon \int_0^t \int_{\mathbb{R}\backslash [-1,1]} z N^{\frac{1}{\epsilon}}(ds,dz),
\end{equation}
where $N^{\frac{1}{\epsilon}}$ is a Poisson random measure defined on a given complete probability space $(\Omega,\mathcal{F}, \mathbb{P})$ with compensator $\epsilon^{-1}ds \otimes \nu$ and $\widetilde{N}^{\frac{1}{\epsilon}}$ is the corresponding compensated Poisson random measure. The L\'evy measure $\nu$ has the following explicit form
\begin{equation*}
  \nu(dz)=\frac{1}{|z|^{1+\alpha}} dz.
\end{equation*}
It clearly satisfies the condition (\ref{Connu}). We know that
\begin{align*}
 \int_0^t \int_{\mathbb{R}\backslash [-1,1]} z dt \nu(dz) & =\int_0^t \int_1^\infty z dt \nu(dz)-\int_0^t \int_1^\infty z dt \nu(dz) \\
 &  = 0.
\end{align*}
So, we can rewrite (\ref{LIDeco}) as
\begin{equation*}
  \epsilon L_{t/\epsilon} = \epsilon \int_0^t \int_{\mathbb{R} \backslash \{0\}} z \widetilde{N}^{\frac{1}{\epsilon}}(ds,dz).
\end{equation*}

Denote the scaled processes $\epsilon L_{t/\epsilon}$ by $\epsilon L_t^\epsilon$. Consider the following SDE,
\begin{equation*}\left\{
  \begin{aligned}
    dY_t^\epsilon & = b(Y_t^\epsilon)dt+\epsilon d L_t^\epsilon, \\
    Y_0^\epsilon & = 0,
    \end{aligned}\right.
\end{equation*}
where $b$ is a bounded Lipschitz continuous function on $[0,1]$. Applying Theorem \ref{LDP} and Theorem \ref{wLDP}, we know that the solution process $Y^\epsilon$ satisfies the large deviations estimates in Theorem \ref{LDP} and the weak large deviation principle presented in Theorem \ref{wLDP} with speed $\log 1/\epsilon$.

\section{Discussion}\label{Dis}

Based on the results for heavy tailed L\'evy process obtained by Rhee et al.~\cite{CH19}, we have obtained a large deviation theorem (Theorem \ref{LDP}) for a kind of one dimensional stochastic differential equations (SDEs) driven by heavy-tailed L\'evy processes. We have also obtained a weak large deviation principle (Theorem \ref{wLDP}) and shown that the full large deviation principle does not hold for such solution process. We require no exponential integrability on the L\'evy measure of the driven process. As a result of that, we have shown that the solution process to the SDEs driven by a symmetric $\alpha$-stable L\'evy process satisfies a large deviation theorem and the weak large deviation principle with speed $\log 1/\epsilon$. This is the first time, to our knowledge, that a large deviation result is proved for SDEs driven by $\alpha$-stable L\'evy process.

Unlike the large deviation principle for SDEs driven by Brownian motion \cite{FW} and random Poisson measure with exponentially light tail \cite{Budhiraja13}, the convergence rate is determined solely by the jump behavior of the driven process and is in a polynomial rate. The large deviation theorem (Theorem \ref{LDP}) shows a more precise convergence rate that varies with respect to sets in the path space, which shares the same idea as  Wang \cite{WangJ}. Even the SDEs in \cite{Gomes} are driven by a subexponential light process, whose tail is much lighter than our driven processes, the probability still decays much slower than an exponential speed. This agrees with what we have shown in our weak large deviation principle, where the probability decays in a polynomial speed. These may offer us with a powerful tool in investigating the dynamical behaviors of SDEs, such as the most probable transition path.

However, we have only obtained the large deviation theorem for bounded sets in Theorem \ref{LDP}, unlike the result in \cite{CH19} which is for all measurable sets. We are only able to derive results for one-dimensional SDEs with additive noise. But our results may still be useful in determining the most probable transition path, as we only need to consider the probability of a tube which is a bounded set in the path space.

Further research is needed to establish the large deviation result for multiplicative noise and for multidimensional settings. There are very few papers concerning SDEs driven by heavy-tailed processes that take into consideration of these two cases. Very few theoretical works deal with the dynamical behaviors of these two cases.

\section*{Appendix: Proofs to Lemma \ref{Conti}, \ref{bda} and \ref{InConti}}\label{profs}
\renewcommand{\theequation}{A.\arabic{equation}}
\renewcommand{\thefigure}{A.\arabic{figure}}
\setcounter{equation}{0}
\setcounter{figure}{0}
\renewcommand{\thetheorem}{A.\arabic{theorem}}

In this section, we provide the proofs to Lemma \ref{Conti}, \ref{bda} and \ref{InConti}. The following lemma will play a role in these proofs.

\begin{lemma}\label{lem}
  For every function $f$ in $\mathbb{D}$ and every series $\{\lambda_n \}_{n \geq 1}$ of non-decreasing homeomorphisms from $[0,1]$ to $[0,1]$, satisfying  the property that $\| \lambda_n-e\| \to 0$ as $n \to \infty$, we have
  \[
  \limsup_{n \to \infty} \int_0^1|f(\lambda_n(s))-f(s))|ds=0.
  \]
\end{lemma}
\begin{proof}
For every $\epsilon>0$, define a set $A^\epsilon=\{s:|f(s)-f(s-)|>\epsilon \}$. As $f \in \mathbb{D}[0,1]$ , we know $A^\epsilon$ is a finite set and $f$ is a bounded function. Let $h_i=|f(s_i)-f(s_i-)|$, $s_i \in A^\epsilon$.

Then we have
$$
\limsup_{n \to +\infty} |f(\lambda_n(s))-f(s)| \leq f^\epsilon(s),
$$
where
$$
f^\epsilon(s)=\left\{ \begin{array}{ll}h_i, & s=s_i\in A^\epsilon, \\
\epsilon, & \textrm {otherwise.}\end{array}\right.
$$
Then
$$
\int_0^t\limsup_{n \to +\infty} |f(\lambda_n(s))-f(s)| ds\leq \int_0^t f^\epsilon(s)ds=t\epsilon.
$$
By Fatou's Lemma, we have
$$
\limsup\limits_{n\to +\infty} \int_0^t|f(\lambda_n(s))-f(s)|ds \leq \int_0^t\limsup_{n \to +\infty} |f(\lambda_n(s))-f(s)| ds\leq \int_0^tf^\epsilon(s)ds=t\epsilon.
$$
Since $\epsilon$ is arbitrary, we get
$$
\limsup\limits_{n \to \infty} \int_0^1 |f(\lambda_n(s))-f(s))|ds =0.
$$
\end{proof}
\begin{proof}[{\bf \em Proof of Lemma \ref{Conti}}]
  Suppose $g_n \to g $ under Skorokhod metric. Then without loss of generality, we assume that $d(g_n,g)<1/n$. This means there exist non-decreasing homeomorphisms $\{\lambda_n \}_{n \geq 1}$ such that $\| \lambda_n -e\| \vee \|g \circ \lambda_n -g_n \|<1/n$.

  Let $f_n=F(g_n)$ and $f=F(g)$. We will prove that $d(f_n,f) \to 0$ as $n \to \infty$.

  For this purpose, we calculate
  \begin{align}
  (f \circ \lambda_n)(t)-f_n(t)&= \int_0^{\lambda_n(t)} b(f(s))ds-\int_0^t b(f_n(s))ds +g(\lambda_n(t))-g_n(t), \label{calc1} \\
           &=\int_0^t b((f\circ \lambda_n)(s))- b(f_n(s))ds +g(\lambda_n(t))-g_n(t)  \notag \\
           &+ \int_0^{\lambda_n(t)} b(f(s))ds - \int_0^{t} b((f \circ \lambda_n )(s))ds.\label{calc2}
  \end{align}
  Note that,
  \begin{align*}
  |(f \circ \lambda_n)(s)-f(s)| & =\left|\int_0^{\lambda_n(s)} b(f(\tau))ds-\int_0^s b(f(\tau))d\tau+g(\lambda_n(s))-g(s) \right| \\
  &=\left | \int_s^{\lambda_n(s)}b(f(\tau))d\tau+g(\lambda_n(s))-g(s) \right|\\
  & \leq C|\lambda_n(s)-s|+|g(\lambda_n(s))-g(s)|.
  \end{align*}
  By Lemma \ref{lem}, we know that there exists a positive sequence $\{\epsilon_n\}_{n \geq 1}$ which tends to $0$ as $n \to \infty$, such that,
  \[
   \int_0^1|g(\lambda_n(s))-g(s)|\leq \epsilon_n .
  \]
  With this, we obtain the asymptotic result for (\ref{calc2}) as
  \begin{align*}
    \left|\int_0^{\lambda_n(t)} b(f(s))ds - \int_0^{t} b((f \circ \lambda_n )(s))ds \right|&=\left|\int_0^t \left[ b(f(s))-b((f \circ \lambda_n)(s))\right]+\int_t^{\lambda_n(t)} b(f(s))ds\right| \\
    & \leq \int_0^t L(C|\lambda_n(s)-s|+|g(\lambda_n(s))-g(s)|)ds+C|\lambda_n(t)-t| \\
    & \leq 2(L \vee 1)C \sup_{0 \leq s \leq 1}|\lambda_n(s)-s|+L \epsilon_n \\
    & \to 0,
  \end{align*}
where $L>0$ is the Lipschitz constant of $b$ and $C$ is the bound of $b$.
  So, continuing (\ref{calc1}), we have
  \begin{align*}
    |(f \circ \lambda_n)(t)-f_n(t)|&=\left|\int_0^t b((f\circ \lambda_n)(s))- b(f_n(s))ds \right|+|g(\lambda_n(t))-g_n(t) | \\
    &  + \left | \int_0^{\lambda_n(t)} b(f(s))ds - \int_0^{t} b((f \circ \lambda_n )(s))ds \right | \\
    &\leq \int_0^t L|(f\circ \lambda_n)(s)- f_n(s)|ds +1/n+\tilde{\epsilon}_n,
  \end{align*}
  where $\tilde{\epsilon}_n = \frac{2(L \vee 1)C}{n} + L\e_n \to 0$ as $n \to \infty$.
  By Gr\"onwall's inequality, we get
  \begin{align*}
    d(f,f_n) & \leq \| \lambda_n -e\| \vee \|f \circ \lambda_n -f_n \| \\
    & \leq 1/n \vee [ (1/n + \tilde{\epsilon}_n) e^L ],
  \end{align*}
  and this leads to our desired result.

  Thus we have proved the continuity of $F$ under Skorokhod metric.
  \end{proof}

\begin{proof}[{\bf \em Proof of Lemma \ref{InConti}}]
Note that the space $F(\mathbb{D})$ is endowed with the same metric as $\mathbb{D}$.

We first prove that $F^{-1}$ is a continuous mapping.
Suppose that, without loss of generality,  there is a sequence $\{f_n\}_{n\geq1}$ and $f$ in $F(\mathbb{D})$ such that $d(f_n,f) \to 0$ as $n \to \infty$. It means that there exists a series of non-decreasing homeomorphisms $\lambda_n$ satisfying $ \|\lambda_n-e \| \vee \|f_n-f \circ\lambda_n \| \to 0$ as $n \to \infty$.

We know that for $f_n$ and $f$ , there are $g_n$ and $g$ in $\mathbb{D}$ such that
$$
\begin{aligned}
f_n(t) & = \int_0^t b(f_n(s))ds +g_n(t), \\
f(t) & = \int_0^t b(f(s))ds +g(t).
\end{aligned}
$$
To prove the continuity of $F^{-1}$, we only need to prove that $d(g_n,g) \to 0$ as $n \to \infty$.

We know that
$$
\begin{aligned}
g_n(t)-g(\lambda_n(t)) &= f_n(t)-f(\lambda_n(t))-\left[\int_0^t b(f_n(s))ds - \int_0^{\lambda_n(t) } b(f(s)) ds\right] \\
& = f_n(t)-f(\lambda_n(t))-\left[\int_0^t\left( b(f_n(s))-b(f(\lambda_n(s))) \right) ds \right] \\
& - \left[\int_0^t b(f(\lambda_n(s)))ds-\int_0^{\lambda_n(t)} b(f(s)) ds \right].
\end{aligned}
$$
We only need to estimate
$$
\int_0^t b(f(\lambda_n(s)))ds-\int_0^{\lambda_n(t)} b(f(s)) ds.
$$
As is done previously,
$$
\begin{aligned}
\int_0^t b(f(\lambda_n(s)))ds-\int_0^{\lambda_n(t)} b(f(s)) ds & = \int_0^t b(f(\lambda_n(s)))-b(f(s)) ds+\int_{\lambda_n(t)}^t b(f(s))ds,
\end{aligned}
$$
By Lemma \ref{lem}, we know that
$$
\limsup\limits_{n\to +\infty}\int_0^1 |f(\lambda_n(s))-f(s))|ds=0.
$$
With this, we obtain that
$$
\begin{aligned}
\left |\int_0^t b(f(\lambda_n(s)))ds-\int_0^{\lambda_n(t)} b(f(s)) ds \right| & \leq \left|\int_0^t b(f(\lambda_n(s)))-b(f(s)) ds\right|+\left|\int_{\lambda_n(t)}^t b(f(s))ds \right| \\
&\leq L \int_0^1|f(\lambda_n(s))-f(s)|ds+C\|\lambda_n -e \| \\
& \to 0,
\end{aligned}
$$
as $n \to \infty$.

So,
$$
\limsup\limits_{n \to \infty}\left \|\int_0^\cdot b(f(\lambda_n(s)))ds-\int_0^{\lambda_n(\cdot)} b(f(s)) ds \right\|=0.
$$
Then, as $n \to \infty$,
$$
\begin{aligned}
\|g_n-g \circ\lambda_n\| &\leq \|f_n-f\circ\lambda_n\|+\left\|\int_0^\cdot b(f_n(s))-b(f(\lambda_n(s)))ds \right\| \\
& + \left\|\int_0^\cdot b(f(\lambda_n(s)))ds-\int_0^{\lambda_n(\cdot)} b(f(s)) ds \right\| \\
& \leq \|f_n-f\circ\lambda_n\| + L \left\|f_n-f\circ \lambda_n \right\|\\
& +\left\|\int_0^\cdot b(f(\lambda_n(s)))ds-\int_0^{\lambda_n(\cdot)} b(f(s)) ds \right\| \\
& \to 0.
\end{aligned}
$$
Thus, we have proved the continuity of $F^{-1}$ under the Skorokhod metric $d$.

Then, we continue to prove that $F^{-1}$ is a Cauchy continuous mapping.

Proof by contradiction. Suppose there exists a Cauchy sequence $\{f_n\}_{n \geq 1}$ in $(F(\mathbb{D},d)$ with $f_n=F(g_n)$ where $g_n \in \mathbb{D}$ and $\{g_n\}_{n \geq 1}$ is not a Cauchy sequence in $(\mathbb{D},d)$.

    We know that  $\{f_n\}_{n \geq 1}$ is also a Cauchy sequence in $(\mathbb{D},d)$. As $(\mathbb{D},d)$ is a complete metric space, there exists a function $f \in \mathbb{D}$ such that $d(f_n,f) \to 0$ as $n \to \infty$. Then, we have non-decreasing homeomorphisms $\{\lambda_n\}_{n \geq 1}$ to make $\|\lambda_n-e \| \vee \| f_n \circ \lambda_n -f \|<1/n$. A direct calculation gives us
    \[
        \|\lambda_n-e \|=\sup_{s \in [0,1]} |\lambda_n(s)-s| =\sup_{s \in [0,1]} |s-\lambda_n^{-1}(s) |<\frac{1}{n}.
    \]
    Denote $\lambda_{n_1} \circ \lambda_{n_2}^{-1}$ as $\tilde{\lambda}_{n_1,n_2}$. It is easy to verify that $\tilde{\lambda}_{n_1,n_2}$ is also a non-decreasing homeomorphism from $[0,1]$ to $[0,1]$.

    Similarly, we have
   \begin{equation} \label{eq6}
    \|\tilde{\lambda}_{n_1,n_2}-e \|=\|\lambda_{n_1} \circ \lambda_{n_2}^{-1} - e \| =\| \lambda_{n_1} - \lambda_{n_2}\| \leq \|\lambda_{n_1}-e \|+ \| \lambda_{n_2} -e \| < \frac{2}{n_1 \wedge n_2},
   \end{equation}
   and
   \begin{equation} \label{eq1}
    \|f_{n_1} \circ \tilde{\lambda}_{n_1,n_2}-f_{n_2} \|=\|f_{n_1}\circ \lambda_{n_1} \circ \lambda_{n_2}^{-1}- f_{n_2} \|  = \|f_{n_1}\circ \lambda_{n_1} - f_{n_2}\circ \lambda_{n_2} \| <\frac{2}{n_1 \wedge n_2 },
   \end{equation}

   for all $n_1,n_2 \in \mathbb{Z}_+$.

   As $\{g_n\}_{n\geq 1}$ is not a Cauchy sequence, there exists a constant $\delta>0$ such that for all $N >0$, we can find an integer $k>0$ satisfying $d(g_N,g_{N+k})>\delta$. For the homeomorphisms $\lambda_n$, this gives us
   \begin{equation}\label{eq2}
    \delta<d(g_N,g_{N+k}) \leq \|g_N \circ \tilde{\lambda}_{N,N+k}-g_{N+k} \|.
   \end{equation}

   We know,
   \begin{align}
    (g_N \circ \tilde{\lambda}_{N,N+k})(t)-g_{N+k}(t) & = (f_N \circ \tilde{\lambda}_{N,N+k})(t)-f_{N+k}(t) \notag \\
       & - \left[ \int_0^{\tilde{\lambda}_{N,N+k}(t)} b (f_N(s)) ds - \int_0^t b(f_{N+k}(s)) ds \right] \notag \\
       & = (f_N \circ \tilde{\lambda}_{N,N+k})(t)-f_{N+k}(t) \label{eq3}\\
       & -\left[ \int_0^t b ( (f_N \circ \tilde{\lambda}_{N,N+k})(s))-b(f_{N+k}(s)) ds \right] \label{eq4}\\
       & + \left[ \int_0^t b ( (f_N \circ \tilde{\lambda}_{N,N+k})(s))ds -\int_0^{\tilde{\lambda}_{N,N+k}(t)} b(f_N(s)) ds \right].\label{eq5}
   \end{align}

   We claim that for each $\epsilon>0$, there exists an integer $N_1>0$ such that for all integer $n_1>N_1$ and positive integer $k$, we have
   \[
    \sup_{t \in [0,1]}\left|\int_0^t b ((f_{n_1} \circ \tilde{\lambda}_{n_1,n_1+k})(s))ds -\int_0^{\tilde{\lambda}_{n_1,n_1+k}(t)} b(f_{n_1}(s)) ds\right| < \epsilon.
   \]

   By the Lipschitz continuity of $b$, for the  previous $\epsilon$, there is an integer $N_2>0$ such that for all $n_2>N_2$ and positive integer $k$, we have
   \[
    \sup_{t \in [0,1]}\left| \int_0^t b ((f_{n_2} \circ \tilde{\lambda}_{n_2,n_2+k})(s))-b(f_{n_2+k}(s)) ds \right| < \epsilon.
   \]
   Take $\epsilon=\frac{1}{2}(\delta-2/n)$. For each $n>N_1 \vee N_2$, there is a positive integer $k$ to make equation (\ref{eq2}) hold with $n$ in place of $N$. We have from (\ref{eq3}-\ref{eq5}) that
   \[
    \delta<\|g_n \circ \tilde{\lambda}_{n,n+k}-g_{n+k} \| \leq \|f_n  \circ \tilde{\lambda}_{n,n+k}-f_{n+k} \|+2\e = \|f_n  \circ \tilde{\lambda}_{n,n+k}-f_{n+k} \|+\delta-2/n.
   \]
   This gives us
   \[
    2/n<\|f_n  \circ \tilde{\lambda}_{n,n+k}-f_{n+k} \|,
   \]
   which contradicts with inequality (\ref{eq1}).

   Now we proceed to prove our claim by showing that for every $k \in \mathbb{Z}_+$,
   \begin{align*}
    \sup_{t \in [0,1]}\left| \int_0^t b ((f_{n} \right. & \left. \circ \tilde{\lambda}_{n,n+k})(s))ds - \int_0^{\tilde{\lambda}_{n,n+k}(t)} b(f_{n}(s)) ds \right| \\
    & \leq \sup_{t \in [0,1]}\left| \int_0^t b((f_{n} \circ \tilde{\lambda}_{n,n+k})(s))-b(f_{n}(s))ds \right|+ C \|\tilde{\lambda}_{n,n+k}-e \| \to 0,
   \end{align*}
   as $n \to \infty$.

   We calculate for $s\in[0,1]$ that
   \begin{align*}
    \left| (f_n\circ \tilde{\lambda}_{n,n+k})(s)-f_n(s) \right| & =\left|(f_{n} \circ \lambda_{n} \circ \lambda_{n+k}^{-1})(s)-f_{n}(s)\right| \\
    & \leq \left|(f_{n} \circ \lambda_{n} \circ \lambda_{n+k}^{-1})(s)-f(\lambda_{n+k}^{-1}(s)) \right|+\left|f(\lambda_{n+k}^{-1}(s)) - f (s)\right|\\
    &+ \left|f(\lambda_{n}^{-1}(s)) - f (s)\right|+\left|f(\lambda_{n}^{-1}(s)) - f_{n} (s)\right|.
   \end{align*}
   As $ \|f_{n} \circ \lambda_{n} -f \| < 1/n$ and $\|\lambda_n^{-1}-e\|=\|\lambda_n-e \|<1/n$, we know
   \begin{align*}
    \sup_{s \in [0,1]} \left|(f_{n} \circ \lambda_{n} \circ \lambda_{n+k}^{-1})(s)-f(\lambda_{n+k}^{-1}(s)) \right| & = \sup_{s \in [0,1]} \left|f(\lambda_{n}^{-1}(s)) - f_{n} (s)\right|\\
    &  = \|f_{n} \circ \lambda_{n} -f \| < \frac{1}{n}.
   \end{align*}
   And for all continuous point $s\in[0,1]$ of $f$ and hence for Lebesgue-almost all $s\in[0,1]$,
   \begin{equation*}
    \left|f(\lambda_{n+k}^{-1}(s)) - f (s)\right| \leq \left|f(\lambda_{n}^{-1}(s)) - f (s) \right| \to 0 \quad \textrm{as } n \to \infty.
   \end{equation*}
   Then, we can deduce that $| (f_n\circ \tilde{\lambda}_{n,n+k})(s)-f_{n}(s)|$ converges to $0$ for Lebesgue-a.e. $s\in[0,1]$ and all $k \in \mathbb{Z}_+$, as $n \to \infty$. Since $b$ is a bounded Lipschitz continuous function, we have from \emph{Dominant Convergence Theorem } that
   \begin{equation} \label{eq7}
    \sup_{t \in [0,1]}\left| \int_0^t b((f_{n} \circ \tilde{\lambda}_{n,n+k})(s))-b(f_{n}(s))ds \right| \leq \int_0^1 \left|b((f_{n} \circ \tilde{\lambda}_{n,n+k})(s))-b(f_{n}(s)) \right| ds \to 0,
   \end{equation}
   as $n \to \infty$.
   A direct calculation gives us
   \begin{align*}
    \left |\int_0^t b ((f_{n} \right. & \left. \circ \tilde{\lambda}_{n,n+k})(s))ds - \int_0^{\tilde{\lambda}_{n,n+k}(t)} b(f_{n}(s)) ds \right| \\
    & \leq \left| \int_0^t b((f_{n} \circ \tilde{\lambda}_{n,n+k})(s))-b(f_{n}(s))ds \right|+ \left| \int_t^{\tilde{\lambda}_{n,n+k})(t)} b(f_n(s))ds \right|.
   \end{align*}
    With the fact that $b$ is a bounded function, inequality (\ref{eq6}) and  (\ref{eq7}), we have proved our claim.
\end{proof}

\begin{proof}[{\bf \em Proof of Lemma \ref{bda}}]
  As is stated in the lemma, we assume that the set $A$ is bounded by a positive constant $M$, by which we mean $\sup_{x\in A} d(x,\mathbf 0) \le M$, where $\mathbf 0 $ is the zero function. We first prove that $F^{-1}(A)$ is bounded away from $\mathbb{D}_{< j,k}$.

  Proof by contradiction. Suppose there exist two sequences $\{\xi_n\}_{n\geq 1} \subset F^{-1}(A) $ and $ \{\eta_n\}_{n\geq 1} \subset \mathbb{D}_{<j,k} $ such that  $d(\xi_n,\eta_n)<\frac{1}{n}$ for all $n \geq 1$. So, there also exists a sequence of non-decreasing homeomorphism $\{\lambda_n\}_{n\geq 1}$, and $\|\lambda_n-e\|\vee \|\xi_n-\eta_n\circ\lambda_n\|<\frac{1}{n}$.

  As $\{F(\xi_n)\}_{n\geq 1} \subseteq A$ is bounded by $M$, by the definition of $F$, which is
  \[
  F(\xi_n)(t)=\int_0^tb(F(\xi_n)(s))ds+\xi_n(t),
   \]
  we have
  \begin{align*}
  d(\xi_n,\mathbf{0})&\leq d(\xi_n,F(\xi_n))+d(F(\xi_n),\mathbf{0}) \\
  &\leq \|\int_0^\cdot b(F(\xi_n)(s))ds \|+M \\
  & \leq C+M,
  \end{align*}
  and
  \begin{equation}\label{bound}
  d(\eta_n,\mathbf{0})\leq d(\eta_n,\xi_n)+d(\xi_n,\mathbf{0})\leq 1+C+M.
  \end{equation}
  These mean that $\{\xi_n\}_{n\geq 1}$ and $ \{\eta_n\}_{n\geq 1}$ are also bounded by $C+M$ and $1+C+M$ respectively.

  Next, we will show that $d(F(\xi_n),F(\eta_n))\to 0$ as $n \to \infty$.

  For the same $\{\lambda_n\}_{n \geq 1}$, we have
  \begin{align*}
   |F(\xi_n)(t)-F(\eta_n)\circ \lambda_n(t)| & \leq \left|\int_0^{\lambda_n(t)}b(F(\eta_n)(s))ds-\int_0^tb(F(\xi_n)(s))ds\right|+|\eta_n \circ \lambda_n(t)-\xi_n(t)|\\
   & \leq \int_0^tL|F(\xi_n)(s)-F(\eta_n)\circ \lambda_n(s)|ds+|\eta_n \circ \lambda_n(t)-\xi_n(t)| \\
   &+\left|\int_0^{\lambda_n(t)}b(F(\eta_n)(s))ds -\int_0^t b(F(\eta_n) \circ \lambda_n(s))ds \right| .
  \end{align*}
  As $\|\eta_n \circ \lambda_n-\xi_n\|<\frac{1}{n}$, we only need to estimate the last term, which is
  \begin{align*}
   \left|\int_0^{\lambda_n(t)}b(F(\eta_n)(s))ds -\int_0^t b(F(\eta_n) \circ \lambda_n(s))ds \right| & \leq \int_0^t| b(F(\eta_n)(s))ds - b(F(\eta_n) \circ \lambda_n(s)) |ds \\
   & +\int_t^{\lambda_n(t)}|b(F(\eta_n)(s))|ds \\
   & \leq \int_0^tL|F(\eta_n)(s)-F(\eta_n) \circ \lambda_n(s)|ds +C\|\lambda_n-e \|.
  \end{align*}

  We claim that
   \[
    \int_0^tL|F(\eta_n)(s)-F(\eta_n) \circ \lambda_n(s)|ds \leq \widetilde{C} \|\lambda_n-e\|,
   \]
  where $\widetilde{C}=LC+4(1+C+M)[(\alpha-1)j+(\beta-1)k]$.

  With this, we can deduce that
  \begin{align*}
   |F(\xi_n)(t)-F(\eta_n)\circ \lambda_n(t)| & \leq \int_0^tL|F(\xi_n)(s)-F(\eta_n)\circ \lambda_n(s)|ds+\|\eta_n \circ \lambda_n-\xi_n\|\\
   &+(C+\widetilde{C})\|\lambda_n-e\| \\
   &\leq \int_0^tL|F(\xi_n)(s)-F(\eta_n)\circ \lambda_n(s)|ds+\frac{(1+C+\widetilde{C})}{n}.
   \end{align*}
  By Gr\"onwall's Lemma, we have
   \[
    |F(\xi_n)(t)-F(\eta_n)\circ \lambda_n(t)| \leq \frac{(1+C+\widetilde{C})}{n} \exp(Lt).
   \]
  As $n \to \infty$, $d(F(\xi_n),F(\eta_n))\to 0$, which contradicts with the assumption that $A$ is bounded away from $\mathbb{G}_{<j,k}$.

  We now continue to verify our claim. Note that
  \begin{align*}
   |F(\eta_n(s))-F(\eta_n) \circ\lambda_n(s)| & \leq \left|\int_{\lambda_n(s)}^sb(F(\eta_n)(\tau))d\tau \right|+|\eta_n(s)-\eta_n \circ \lambda_n(s)|\\
   & \leq C\|\lambda_n-e\|+|\eta_n(s)-\eta_n \circ \lambda_n(s)|,
  \end{align*}
  and
   \[
    |\eta_n(s)-\eta_n \circ \lambda_n(s)| \leq \zeta_n(s),
   \]
  where
   \[
    \zeta_n=
    \begin{cases}
     2(1+C+M), &   s\in \bigcup_{i}\big[0\vee(s_i^{(n)}-\|\lambda_n-e\|), (s_i^{(n)}+\|\lambda_n-e\|)\wedge1 \big], \\
    0, & \textrm{otherwise,}
    \end{cases}
   \]
   and $s_i^{(n)}$'s are the jump points of $\eta_n$. This is because $\eta_n$ is a step function with at most $((\alpha-1)j+(\beta-1)k)$ jumps, and  the jump heights of $\eta_n$ are bounded by $2(1+C+M)$ due to (\ref{bound}).

  Thus,
  \begin{align*}
    \int_0^tL|F(\eta_n)(s)-F(\eta_n) \circ \lambda_n(s)|ds &\leq \int_0^t LC\|\lambda_n-e\|ds+\zeta_n(s) ds \\
    & \leq LC\|\lambda_n-e\|+4(1+C+M)[(\alpha-1)j+(\beta-1)k]\|\lambda_n-e\| \\
   &=\widetilde{C}\|\lambda_n-e\|.
  \end{align*}

  Replacing $\mathbb{D}_{<j,k}$ by $\mathbb{D}_{<j,k} \cup \mathbb{D}_{j,k}$ and $\mathbb{G}_{<j,k}$ by $\overline{F(\mathbb{D}_{<j,k}\cup\mathbb{D}_{j,k})}$, we can deduce the second statement with exactly the same proof and it is omitted here.
\end{proof}

\paragraph{Acknowledgements.}
The research of W. Wei was partly supported by the NSFC grant 11771449. The research of Q. Huang is supported by FCT, Portugal, project PTDC/MAT-STA/28812/2017. We would like to thank Shenglan Yuan for helpful discussions.


\begin{thebibliography}{10}

\bibitem{WangJ}
G. Ben Arous and J. Wang,
\newblock Very rare events for diffusion processes in short time,
\newblock arXiv:1901.10025v1

\bibitem{Dupuis}
M. Bou\'{e} and P. Dupuis,
\newblock A variational representation for certain functionals of Brownian motion
\newblock {\em Ann. Probab.}, 26 (1998), 1641-1659.

\bibitem{Budhiraja13}
A. Budhiraja, J. Chen and P. Dupuis,
\newblock Large deviations for stochastic partial differential equations driven by a Poisson random measure,
\newblock {\em Stoch. Process. Their Appl.}, 123 (2013), 523-560.

\bibitem{Budhiraja08}
A. Budhiraja, P. Dupuis, and V. Maroulas,
\newblock Large deviations for infinite dimensional stochastic dynamical systems
\newblock {\em Ann. Probab.}, 36 (2008), 1390-1420.

\bibitem{Budhiraja11}
A. Budhiraja, P. Dupuis, and V. Maroulas,
\newblock Variational representations for continuous time processes,
\newblock {\em Ann. Inst. H. Poincaré Probab. Statist.}, 47 (2011), 725-747.

\bibitem{Gomes}
A. de Oliveira Gomes,
\newblock Large Deviations Studies for Small Noise Limits of Dynamical Systems Perturbed by L\'evy Processes,
\newblock {\em Dissertation}, Humboldt-Universität zu Berlin, Mathematisch-Naturwissenschaftliche Fakultät, 2018, http://dx.doi.org/10.18452/19118.

\bibitem{Del2}
D. del Castillo-Negrete,
\newblock  Non-diffusive, non-local transport in fluids and plasmas,
\newblock {\em Nonlinear Process}. Geophys., 17 (2010), 795–807.

\bibitem{Del1}
D. del Castillo-Negrete, V. Yu. Gonchar, and A. V. Chechkin,
\newblock Fluctuation-driven directed transport in the presence of L\'evy flights,
\newblock {\em Phys. A}, 27 (2008), pp. 6693–6704.

\bibitem{Demb}
A. Dembod and O. Zeltouni,
\newblock {\em Large Deviations Techniques and Applications},
\newblock 38. Springer-Verlag New York, Inc, 1998.

\bibitem{DuanS}
J. Duan,
\newblock {\em An Introduction to Stochastic Dynamics,}
\newblock Cambridge University Press, 2015

\bibitem{Weinan}
W. E, W. Ren and E. Vanden-Eijnden,
\newblock Minimum action method for the study of rare events,
\newblock {\em Commun. Pure Appl. Math.}, Vol. LVII (2004), 0637–0656.

\bibitem{FW}
M. I. Freidlin and A. D. Wentzell,
\newblock {\em Random Perturbations of Dynamical Systems, third edition},
\newblock Springer, Berlin, 2012

\bibitem{Ganguly18}
A. Ganguly,
\newblock Large deviations principle for stochastic integrals and stochastic differential equations driven by infinite-dimensional semimartingales,
\newblock {\em Stoch. Process. Their Appl.}, 128 (2018):2179--2227, 2018.

\bibitem{Ting}
T. Gao, J. Duan , X. Li and R. Song,
\newblock  Mean exit time and escape probability for dynamical systems driven by L\'evy noises,
\newblock {\em SIAM J. Sci. Comput.}, 36 (2014), A887–A906.

\bibitem{Henry}
P. Henry-Labordère,
\newblock Counterparty risk valuation: a marked branching diffusion approach (2012).
\newblock arXiv:1203.2369

\bibitem{Qiao}
Q. Huang, W. Wei and J. Duan,
\newblock Large Deviations for Stochastic Differential Equations Driven by Semimartingales,
\newblock arXiv:1910.05720v1.

\bibitem{Imkeller}
P. Imkeller, I. Pavlyukevich, and T. Wetzel,
\newblock  First exit times for L\'evy-driven diffusions with exponentially light jumps,
\newblock {\em Ann. Probab.}, 37 (2009), 530-564.

\bibitem[Lee(2013)]{Lee13}
J. M. Lee.
\newblock \emph{Introduction to smooth manifolds, second edition}, Vol. 218.
\newblock Springer Science+Business Media New York, 2013.

\bibitem{Lindskog}
F. Lindskog, S. I. Resnick,and J. Roy,
\newblock  Regularly varying measures on metric spaces: Hidden regular variation and hidden jumps,
\newblock {\em Probab. Surv.}, 11(2014), 270–314.

\bibitem{Oksendal}
B. {\O}ksendal,
\newblock {\em Stochastic Differential Equations. An Introduction with Applications}. \newblock Universitext. Springer, Berlin (1985)

\bibitem{priola2012}
E. Priola,
\newblock Pathwise uniqueness for singular SDEs driven by stable processes.
\newblock {\em Osaka J. Math.}, 49(2):421--447, 2012.

\bibitem{CH19}
C.H. Rhee, J. Blanchet and B. Zwart,
\newblock Sample path large deviations for Lévy processes and random walks with regularly varying increments.
\newblock {\em Ann. Probab.}, 47(6):3551--3605, 2019.

\bibitem{Skorokhod}
A. V. Skorokhod,
\newblock  Branching diffusion processes.
\newblock {\em Teor. Verojatnost. i Primenen}. 9, 492–497 (1964)

\bibitem{Wan}
X. Wan,
\newblock  A minimum action method for small random perturbations of two-dimensional parallel shear slows,
\newblock {\em J. Comput. Phys.}, 235 (2013), 497–514

\bibitem{Woy}
W. A. Woyczynski,
\newblock {\em L\'evy processes in the Physical Sciences},
\newblock in L\'evy Processes: Theory and Applications, O. E. Barndorff-Nielsen, T. Mikosch, and S. I. Resnick, eds., Birkha¨user, Boston, 2001, Birkhauser, pp. 241–266.

\end{thebibliography}
\end{document}